\documentclass{amsart}
\usepackage{amssymb,amsthm,amsmath}
\usepackage{xypic}

\newcommand{\csp}{\mathcal{P}(\omega)}
\newcommand{\F}{\mathcal{F}}
\newcommand{\csf}{{}\sp{\omega}2}
\newcommand{\U}{\mathcal{U}}
\newcommand{\B}{\mathcal{B}}
\newcommand{\A}{\mathcal{A}}
\newcommand{\V}{\mathcal{V}}
\newcommand{\Open}{\mathcal{O}}
\newcommand{\Sone}[1]{\mathrm{S}_1(#1)}
\newcommand{\Sfin}[1]{\mathrm{S}_{\mathrm{fin}}(#1)}
\newcommand{\Ufin}[1]{\mathrm{U}_{\mathrm{fin}}(#1)}
\newcommand{\up}[1]{{#1}\!\!\uparrow}
\newcommand{\pair}[1]{\langle #1 \rangle}
\renewcommand{\d}{\mathfrak{d}}
\renewcommand{\b}{\mathfrak{b}}

\newcommand{\baire}{{}\sp\omega{\omega}}
\newcommand{\cont}{\mathfrak{c}}
\newcommand{\I}{\mathcal{I}}
\newcommand{\X}{\mathcal{X}}
\newcommand{\Y}{\mathcal{Y}}
\newcommand{\W}{\mathcal{W}}
\renewcommand{\S}{\mathcal{S}}

\newtheoremstyle{theorem}
     {11pt}
     {11pt}
     {}
     {}
     {\bfseries}
     {}
     {.5em}
     {\noindent\thmnumber{#2}. \thmname{#1}\thmnote{#3}}

\theoremstyle{theorem}

\newtheorem{lemma}{Lemma}[section]
\newtheorem{propo}[lemma]{Proposition}
\newtheorem{coro}[lemma]{Corollary}

\newtheorem{thm}[lemma]{Theorem}
\newtheorem{ques}[lemma]{Question}

\title{Some observations on filters with properties defined by open covers}
\author[R. Hern\'andez-Guti\'errez]{Rodrigo Hern\'andez-Guti\'errez}
\email[R. Hern\'andez-Guti\'errez]{rodhdz@yorku.ca}
\author[P. J. Szeptycki]{Paul J. Szeptycki}
\email[P. J. Szeptycki]{szeptyck@yorku.ca}

\address{Department of Mathematics and Statistics, York University, Toronto, ON M3J 1P3, Canada}

\date{\today}

\begin{document}

\begin{abstract}
We study the relation between the Hurewicz and Menger properties of filters considered topologically as subspaces of $\mathcal{P}(\omega)$ with the Cantor set topology.
\end{abstract}

\maketitle

\section{Introduction}

A filter $\F$ on a non-empty set $X$ is a subset $\F\subset\mathcal{P}(X)$ such that: (a) $\emptyset\notin\F$, (b) if $x,y\in\F$ then $x\cap y\in\F$, and (c) if $x\in\F$ and $x\subset y\subset X$, then $y\in\F$. Since the power set $\csp$ can be identified with the Cantor set $\csf$ via characteristic functions, a filter on a countable set can be thought of as a subspace of the Cantor set. The topology of filters on $\omega$ has recently attracted much interest (see \cite{marciszewski}, \cite{hg-hr-CDH} and \cite{kunen-medini-zdomsky-Pfilters}). Also, in Chapter 4 of \cite{bart} it is possible to find an extensive study of measure and category theoretic properties of filters.

In this note we would like to study covering properties of filters defined via the selection principles introduced in \cite{COC2}. Henceforth, we restrict our discussion to Lindel\"of topological spaces. For such a topological space $X$, let $\Open$ be the collection of open covers of $X$. A cover $\U\in\Open$ is a \emph{$\omega$-cover} if $X\notin\U$ and for every $S\in[X]\sp{<\omega}$ there is $U\in\U$ with $S\subset\U$. A cover $\U\in\Open$ is a \emph{$\gamma$-cover} if it is infinite and for every $x\in X$ the set $\{U\in\U:x\not\in U\}$ is finite. The collection of $\omega$-covers and $\gamma$-covers are denoted by $\Omega$ and $\Gamma$, respectively. Clearly, $\Gamma\subset\Omega\subset\Open$. Given families $\A$ and $\B$ of open covers and a topological space $X$, consider the following definitions from \cite{COC2}:
\begin{itemize} 
\item $X$ is $\Sone{\A,\B}$ if given $\{\U_n:n<\omega\}\subset\A$, for each $n<\omega$ there is $U_n\in\U_n$ such that $\{U_n:n<\omega\}\in\B$, 
\item $X$ is $\Sfin{\A,\B}$ if given $\{\U_n:n<\omega\}\subset\A$, for each $n<\omega$ there is $\V_n\in[\U_n]\sp{<\omega}$ such that $\bigcup\{\V_n:n<\omega\}\in\B$,
\item $X$ is $\Ufin{\A,\B}$ if given $\{\U_n:n<\omega\}\subset\A$, for each $n<\omega$ there is $\V_n\in[\U_n]\sp{<\omega}$ such that $\{\bigcup\V_n:n<\omega\}\in\B$, 
\end{itemize}

This provides $3\times3\times3=27$ possibly different classes of topological spaces. Some of these properties are null and some reduce to others, see figures 1, 2 and 3 of \cite{COC2} or figure 1 of \cite{tsaban-open_problems_2}. The final diagram obtained is the following, which is called the Scheepers diagram.

$$
\xymatrix{
 & & \Ufin{\Open,\Gamma} \ar[r] & \Ufin{\Open,\Omega} \ar[rr] & & \Sfin{\Open,\Open}\\
 & & \Sfin{\Gamma,\Omega} \ar[ru] & \\
\Sone{\Gamma,\Gamma} \ar[rruu] \ar[r] & \Sone{\Gamma,\Omega} \ar[ru] \ar[rr] & \ar[u] & \Sone{\Gamma,\Open} \ar[rruu] & & \\
 & & \Sfin{\Omega,\Omega} \ar@{-}[u] & \\
\Sone{\Omega,\Gamma} \ar[uu] \ar[r] & \Sone{\Omega,\Omega} \ar[rr] \ar[ur] \ar[uu] & & \Sone{\Open,\Open} \ar[uu]
}
$$

As we will see in the next section, the strongest possible properties that a filter may satisfy are $\Sfin{\Omega,\Omega}$ and $\Ufin{\Open,\Gamma}$. In particular, it is impossible for a filter to satisfy any of the 6 properties at the front of the 3-dimensional Scheepers diagram. Also, following standard terminology we have the following 
\begin{itemize}
\item $X$ is called Menger if $X$ is $\Sfin{\Open,\Open}$, and
\item $X$ is called Hurewicz if $X$ is $\Ufin{\Open,\Gamma}$.
\end{itemize}

 In Section 3, existence of filters with these properties will be established, and in the final section we show that the FUF-filters of Reznichenko and Sipacheva always have the Hurewicz property. 

See \cite{tsaban-menger_hurewicz_book} for a history of the Menger and Hurewicz properties and \cite{tsaban-zdomskyy-scales_fields_hurewicz} for the current status of existence of sets of reals with these properties. It is also interesting that the Menger property has been used to prove the existence of a non countable dense homogeneous filter (see \cite{CDH-menger}). 

\section{Reductions and relations}

From this point on, we will restrict our attention to filters defined on countable sets (and most of the times, on $\omega$) that contain the Frech\'et filter of cocountable sets. Recall that a set $\I\subset\mathcal{P}(X)$ is an ideal if $\I\sp\ast=\{A\subset X: X\setminus A\in\I\}$ is a filter. In this situation, we say that the ideal $\I$ and filter $\I\sp\ast$ are dual. Moreover, the function that takes each set in $\mathcal{P}(X)$ to its complement is a homeomorphism so in fact an ideal is homeomorphic to its dual filter. Thus, we may sometimes talk about ideals instead of filters when the ideal description is simpler.

For each $Y\subset\omega$, let $\up{Y}=\{X\subset\omega:Y\subset X\}$. The set $\up{Y}$ is always closed, if $Y$ is in a filter $\F$ then $\up{Y}$ is contained in $\F$ and if $Y$ is finite then $\up{Y}$ is also open.

Starting from Scheepers's diagram, we first rule out some properties that behave in a trivial way when we restrict to filters. Let us enumerate some known results about separable and metrizable spaces that we will use in our analysis. 

\begin{lemma}\label{folklore}
\begin{itemize}
\item[(i)] \cite[Theorem 2.2]{COC2} Every $\sigma$-compact set is Hurewicz.
\item[(ii)] \cite[Theorem 3.1]{COC2} If $X$ is Menger (Hurewicz) and $Y\subset X$ is closed, then $Y$ is Menger (Hurewicz, respectively) as well.
\item[(iii)] If $X$ is Menger (Hurewicz) and $K$ is $\sigma$-compact, then $X\times K$ is also Menger (Hurewicz, respectively).
\item[(iv)] \cite[Theorem 3.1]{COC2} The continuous image of a Menger (Hurewicz) space is also Menger (Hurewicz, respectively).
\item[(v)] \cite[p. 255]{COC2} If a space is the countable union of Menger (Hurewicz) spaces, then it is Menger (Hurewicz) as well.
\item[(vi)] The space $\baire$ is not Menger.
\item[(vii)] \cite[Theorem 2.3]{COC2} The Cantor set is not $\Sone{\Gamma,\Open}$.
\end{itemize}
\end{lemma}

\begin{lemma}\cite[Theorem 3.9]{COC2}\label{sfinomegaomega}
A separable metrizable space $X$ is $\Sfin{\Omega,\Omega}$ if and only if $X\sp{n}$ is Menger for each $0<n<\omega$.
\end{lemma}

Now let $\F$ be a filter on $\omega$. If $\F$ is not the Frech\'et filter, then there is $x\in\F$ that is coinfinite. Then $\up{x}$ is a closed copy of the Cantor set contained in $\F$. Thus, by (vii) in Lemma \ref{folklore}, $\F$ is not $\Sone{\Gamma,\Open}$. With this observation, six of the properties of the Scheepers diagram are impossible for filters. Moreover, we have the following observation.

\begin{propo}\cite[Claim 20]{chod_rep_zdo-mathias_combinatorial_filters}
If a filter $\F$ is Menger (Hurewicz) then for all $0<n<\omega$, $\F\sp{n}$ is Menger (Hurewicz, respectively).
\end{propo}
\begin{proof}
For $0<n<\omega$, let $\phi:\F\times\csp\sp{n}\to\csp\sp{n}$ be given by $\phi(\pair{F,x_0,\ldots,x_{n-1}})=\pair{F\cup x_0,\dots,F\cup x_{n-1}}$. Then $\phi$ is continuous and its image is $\F\sp{n}$. Then the result follows by Lemma \ref{folklore}.
\end{proof}

Thus, we are left with essentially two non-trivial properties of filters: Menger and Hurewicz. Let us start out by considering some simple results obtained from cardinality considerations. For $f,g\in\baire$, denote $f\leq\sp\ast g$ if and only if there is $N<\omega$ such that if $N\leq n<\omega$, then $f(n)\leq g(n)$. A family $B\subset\baire$ is (i) bounded if there is $g\in\baire$ with $f\leq\sp\ast g$ for every $f\in B$, (ii) dominating if for every $f\in\baire$ there is $g\in B$ with $f\leq\sp\ast g$. Recall the following classic result by Hurewicz.

\begin{propo}\cite{hur1927}\label{classichurewicz}
Let $X$ be a $0$-dimensional separable metric space.
\begin{itemize}
\item[(a)] $X$ is Menger if and only if every continuous image of $X$ in $\baire$ is not dominating.
\item[(b)] $X$ is Hurewicz if and only if every continuous image of $X$ in $\baire$ is bounded.
\end{itemize}
\end{propo}

Thus, $\b$ is the size of the smallest non-Hurewicz set of reals and $\d$ is the size of the smallest non-Menger set of reals (see \cite{vd62} for an introduction to small cardinal numbers). Now, all (non-Frech\'et) filters are of size $\cont$ but there is still a natural and meaningful substitute for cardinality. Recall that a subset $\B$ of a filter $\F$ is a base (of $\F$) if for all $x\in\F$ there is $y\in\B$ with $y\subset x$.

\begin{lemma}\label{lemmabase}
\begin{itemize}
\item[(a)] Let $\F$ be a filter with base $\B$. If $\B$ is Menger (Hurewicz) then $\F$ is also Menger (Hurewicz, respectively).
\item[(b)] If $\B$ is a set with the finite intersection property and every finite power of $\B$ is Menger (Hurewicz), then the filter generated by $\B$ is also Menger (Hurewicz, respectively).
\end{itemize}
\end{lemma}
\begin{proof}
For (a), let $\phi:\B\times\csp\to\csp$ given by $\phi(\pair{x,y})=x\cup y$. Then the conclusion follows from Lemma \ref{folklore} and the fact that $\phi$ is a continuous function with image $\F$.

Now, let $\B$ have the finite intersection property. For each $n<\omega$, let $\phi_{n}:\B\sp{n}\to\B$ be defined by $\phi_{n}(\pair{x_0,\ldots,x_{n-1}})=x_0\cap\ldots\cap x_{n-1}$. Then notice that the filter generated by $\B$ has base $\bigcup\{\phi\sp{n}[\B\sp{n}]:n<\omega\}$ so by part (a) and Lemma \ref{folklore}, we are done.
\end{proof}

\begin{propo}\label{lessthanb}
$\b$ is the minimal character of a filter that is not Hurewicz.
\end{propo}
\begin{proof}
The fact that all filters of character $<\b$ are Hurewicz follows directly from Lemma \ref{lemmabase} and the fact that any set of size $<\b$ is Hurewicz. Now we construct a non-Hurewicz filter of size $\b$.

Let $\{f_\alpha:\alpha<\b\}$ be an unbounded family wrt $\leq\sp\ast$. Let $\F$ be the filter on $\omega\times\omega$ generated by the sets of the form $F_\alpha=\{\pair{n,m}:m\geq f_\alpha(n)\}$ with $\alpha<\b$.

For each $n,m<\omega$, let $U(n,m)=\{x\subset\omega\times\omega:m=\min\{k<\omega:\pair{n,k}\in x\}\}$. Then $\U_n=\{U(n,m):m<\omega\}$ is an open cover of $\F$ for each $n<\omega$. 

Assume there is a $\gamma$-cover of $\F$ of the form $\U=\{U(n,0)\cup\ldots\cup U(n,g(n)):n<\omega\}$ for some $g:\omega\to\omega$. Let $\beta<\b$ be such that $f_\beta\not\leq\sp\ast g$. Then there is a set $B\in[\omega]\sp\omega$ such that $g(n)<f_\beta(n)$ for all $n<\omega$. Then $F_\beta\notin U(n,m)$ for any $n\in B$ and $m<g(n)$. This implies that $\U$ is not a $\gamma$-cover. This contradiction proves that $\F$ is not Hurewicz.
\end{proof}

Using almost the same proof we have the following.

\begin{propo}\label{lessthand}
$\d$ is the minimal character of a filter that is not Menger.
\end{propo}

Thus, filters of small character are trivial examples of filters that are Menger or Hurewicz. Another trivial way of obtaining Menger or Hurewicz examples of filters is by considering $F_\sigma$ filters ((i) in Lemma \ref{folklore}). For example, every countably generated filter is $F_\sigma$. See \cite{hru-combinatorics_filters_ideals} for examples of $F_\sigma$ ideals. Then we would like to find filters that are Menger or Hurewicz, that are not $F_\sigma$ and such that their characters are at least $\d$ or $\b$, respectively. The interesting question we leave open is the following.

\begin{ques}\label{mengernonhurewicz}
Does there exist a Menger filter of character $\d$ that is not Hurewicz?
\end{ques}

Before starting to give consistent answers to the two questions above, we would like to mention a relation between filters and a certain notion of forcing.  For every filter $\F$ on $\omega$, $\mathbb{M}_\F$ denotes the Mathias forcing with respect to $\F$. It is shown on \cite{chod_rep_zdo-mathias_combinatorial_filters} that some properties of this forcing notion are equivalent to $\F$ being either Menger or Hurewicz. 

It turns out that Menger filters are those called Canjar filters in \cite{canjar1} and Hurewicz filters are called strongly Canjar in \cite{canjar2}. In Proposition 3 of \cite{canjar1}, a combinatorial characterization of Canjar filters is given.  Also, there are constructions of Canjar MAD families, Canjar ideals and Canjar ultrafilters using additional assumptions.

Since ultrafilters are non-meager (\cite[Theorem 4.1.1]{bart}), it follows from Proposition \ref{classichurewicz} that there are no Hurewicz ultrafilters. This is mentioned in \cite[Corollary 14]{canjar2} in the language of Canjar filters. According to \cite[Proposition 2]{canjar1}, $\d=\cont$ implies there is a Canjar (thus, Menger) ultrafilter. Moreover, the characterization from \cite{canjar1} mentioned above implies that a Menger ultrafilter is a $P$-point. Thus, Menger ultrafilters consistently do not exist.

Finally, let us recall that the well-known Hurewicz theorem states that if $X\subset\csp$ is an analytic set, then either $X$ is $F_\sigma$ or $X$ contains a closed copy of $\baire$ (\cite[21.18]{kechris}). Since the Menger property is hereditary to closed sets and $\baire$ is not Menger (Lemma \ref{folklore}), it follows that the examples we are looking for must be non-analytic. 

\section{Menger and Hurewicz examples}

A relatively easy way to define a Menger filter in ZFC is as follows: Given a set $X\subset\csf$, we define the ideal $\I_X$ to be the ideal (in the countable set $2\sp{<\omega}$) generated by finite unions of branches $C_x=\{x\!\!\restriction_{n}:n<\omega\}$ where $x\in X$. This definition is due to Nyikos \cite{nyikos-cantortree_fu}. Also, denote by $\F_X$ the filter dual to $\I_X$. If $t\in 2\sp{<\omega}$, let $\pair{t}=\{x\in\csf:t\subset x\}$.

\begin{lemma}\label{refining}
Let $\U$ be an $\omega$-cover of some set $X\subset\csf$. Then there exists a countable $\omega$-cover $\V$ that refines $\U$ and such that its elements are sets of the form $\pair{t_0}\cup\ldots\cup\pair{t_k}$ such that $\{t_0,\ldots,t_k\}\subset 2\sp{m}$ for some $m<\omega$.
\end{lemma}
\begin{proof}
Fix $n<\omega$ for the moment. For each $x=\pair{x_0,\ldots,x_{n-1}}\in X\sp{n}$, choose any $U(x)\in\U$ such that $\{x_0,\ldots,x_{n-1}\}\subset U(x)$. Then there is $m<\omega$ such that $\pair{x_i\!\!\restriction_{m}}\subset U(x)$ for all $i<n$. Let $V(x)=\bigcup\{\pair{x_i\!\!\restriction_{m}}:i<n\}$. Then the collection $\{V(x)\sp{n}:x\in X\sp{n}\}$ is an open cover of $X\sp{n}$. Thus, there is a countable collection $\V_n\subset\{V(x):x\in X\sp{n}\}$ such that the $n$th powers of $\V_n$ cover $X\sp{n}$. The cover we are looking for is $\V=\bigcup\{\V_n:n<\omega\}$.
\end{proof}

An open cover $\U$ of a space $X$ is said to be a groupable $\omega$-cover if there is a partition $\U=\bigcup\{\U_n:n<\omega\}$ where each $\U_n$ is finite and for all $F\in[X]\sp{<\omega}$, for all but finitely many $n<\omega$ there is $U\in\U_n$ with $F\subset U$. The class of groupable $\omega$-covers is denoted by $\Omega\sp{\rm gp}$.

\begin{lemma}\cite{COC7}\label{prodhurewicz}
A separable metrizable space $X$ has all its finite powers Hurewicz if and only if it is $\Sfin{\Omega,\Omega\sp{\rm gp}}$.
\end{lemma}

In the following result, item (b) can be essentially obtained from \cite[Proposition 14]{canjar1} and \cite{chod_rep_zdo-mathias_combinatorial_filters}. However, we include a direct prove that uses only topological notions.

\begin{thm}\label{branches}
Let $X\subset\csf$. Then
\begin{itemize}
\item[(a)] $X$ is $F_\sigma$ if and only if $\F_X$ is $F_\sigma$,
\item[(b)] every finite power of $X$ is Menger if and only if $\F_X$ is Menger, and
\item[(c)] every finite power of $X$ is Hurewicz if and only if $\F_X$ is Hurewicz.
\end{itemize}
\end{thm}
\begin{proof}
First, consider $[X]\sp{<\omega}$ with the Vietoris topology and define the set
$$
\mathcal{G}=\big\{\pair{F,Y}\in[X]\sp{<\omega}\times\mathcal{P}(2\sp{<\omega}):Y\subset\bigcup\{C_x:x\in F\}\big\}
$$

Assume that $\pair{F,Y}\in([X]\sp{<\omega}\times\mathcal{P}(2\sp{<\omega}))\setminus\mathcal{G}$. Then there is $s\in Y$ with $s\notin\bigcup\{C_x:x\in F\}$. Then $\{\pair{G,Z}:G\subset X\setminus\pair{s},s\in Z\}$ is an open set that contains $\pair{F,Y}$ and is disjoint from $\mathcal{G}$. Thus, $\mathcal{G}$ is closed in $[X]\sp{<\omega}\times\mathcal{P}(2\sp{<\omega})$. Notice that $\I_X$ is the projection of $\mathcal{G}$ into the second coordinate. Then by Lemma \ref{folklore} we have that $\F_X$ has each of the properties in turn if every finite power of $X$ has the corresponding one.

Now assume that $\I_X$ is $F_\sigma$. Then according to Mazur's theorem (see \cite[Theorem 1.5]{hru-combinatorics_filters_ideals}) there exists a lower semicontinuous submeasure $\phi:\mathcal{P}(2\sp{<\omega})\to[0,\infty]$ with $\I_X=\{A\subset 2\sp{<\omega}:\phi(A)<\infty\}$. This means that (a) $\phi(\emptyset)=0$, (b) $\phi(A)\leq\phi(B)$ if $A\subset B$, (c) $\phi(A\cup B)\leq\phi(A)+\phi(B)$, (d) $\phi(F)<\infty$ if $F\in[2\sp{<\omega}]\sp{<\omega}$ and (e) $\phi(A)=\lim_{n\to\infty}{\phi(A\cap2\sp{n})}$. For $n<\omega$, define $U_n=\bigcup\{\pair{s}:\phi(\{s\!\!\restriction_m:m\leq|s|\})>n\}$, this is an open set of $\csf$. Then it is enough to notice that $X=\csf\setminus\bigcap\{U_n:n<\omega\}$. Thus, (a) holds.

In order to prove (b), according to Lemma \ref{sfinomegaomega} it is enough to prove that $X$ is $\Sfin{\Omega,\Omega}$ assuming that $\F_X$ is Menger. Let $\{\U_n:n<\omega\}$ be a sequence of $\omega$-covers of $X$. By Lemma \ref{refining}, we may assume that for each $U\in\U_n$ there is $m(U)<\omega$ and $\{t(U,0),\ldots,t(U,k_U)\}\subset 2\sp{m(U)}$ such that $U=\bigcup\{\pair{t(U,i)}:i\leq k_U\}$.

Fix $n<\omega$ for a moment. For each $U\in\U_n$, define 
$$
x_U=2\sp{m(U)}\setminus\{t(U,0),\ldots,t(U,k_U)\}.
$$
We claim that $\V_n=\{\up{x_U}:U\in\U_n\}$ is a cover of $\F_X$. If $F\in\F_X$, then there are $x_0,\ldots,x_k\in X$ such that $2\sp{<\omega}\setminus (C_{x_0}\cup\ldots\cup C_{x_k})\subset F$. Let $U\in\U_n$ be such that $\{x_0,\ldots,x_k\}\subset U$. It then follows that $2\sp{m(U)}\cap(C_{x_0}\cup\ldots\cup C_{x_k})\subset\{t(U,0),\ldots,t(U,k_U)\}$. Thus $x_U\subset 2\sp{<\omega}\setminus (C_{x_0}\cup\ldots\cup C_{x_k})\subset F$, which implies that $F\in \up{x_U}$.

Thus, since $\F_X$ is Menger, for each $n<\omega$, there is $\S_n\in[\U_n]\sp{<\omega}$ such that $\bigcup\{\up{x_U}:U\in\S_n,n<\omega\}$ is a cover of $\F_X$. Now it is easy to see that $\bigcup\{\S_n:n<\omega\}$ is an $\omega$-cover of $X$. This shows (b).

To prove (c), we will use the characterization in Lemma \ref{prodhurewicz}. So assume that $\F_X$ is Hurewicz. Using the same notation as in the proof of (b), let $\{\U_n:n<\omega\}$ be a sequence of $\omega$-covers of $X$. Then following the proof, since $\F_X$ is Hurewicz, for each $n<\omega$, there is $\S_n\in[\U_n]\sp{<\omega}$ such that $\{\bigcup\{\up{x_U}:U\in\S_n\}: n<\omega\}$ is a $\gamma$-cover of $\F_X$. Then it is easy to see that $\V=\bigcup\{\S_n:n<\omega\}$ is a groupable $\omega$-cover of $X$ with the grouping given precisely by $\{\S_n:n<\omega\}$.
\end{proof}

In \cite{bart-tsaban}, it was proven that there is a set of reals of size $\b$ with all their finite powers Hurewicz, hence 

\begin{coro}
(ZFC) There exists a Hurewicz filter of character $\b$ that is not $F_\sigma$.
\end{coro}

However, it is still open if there is a set of reals of size $\d$ with all their powers Menger (Problem 3.2 from \cite{tsaban-open_problems_2}). Thus, we have the following.

\begin{ques} Is there, assuming only ZFC, a Menger filter of character $\d$ that is not $F_\sigma$?
\end{ques}

Chaber and Pol have shown that if $\b=\d$, then there exists a non-Hurewicz set of reals with all its finite powers Menger. A proof of the Chaber and Pol result can be found in \cite{tsaban-menger_hurewicz_book}.

Using Theorem \ref{branches}, it follows that under $\b=\d$ there is a Menger filter that is not Hurewicz and has character $\d$. However, it turns out that the Chaber-Pol construction can be slightly modified to give a direct construction of a filter. For the sake of completeness, we will outline how to show this. 
%See also \cite{rinot-separating_menger_hurewicz}.

\begin{propo}\label{bequalsd}
If $\b=\d$ there exists a Menger filter of character $\d$ that is not Hurewicz.
\end{propo}
\begin{proof}
Recall that $[\omega]\sp\omega$ can be considered a subspace of $\baire$ by sending each set to its enumerating function. First, we need a collection $\{x_\alpha:\alpha<\b\}$ of infinite, coinfinite subsets of $\omega$ such that  
\begin{quote}
\noindent{$(\ast)$} for every $x\in[\omega]\sp\omega$, $|\{\alpha<\b: x_\alpha\leq\sp\ast x\textrm{ or }\omega\setminus x_\alpha\leq\sp\ast x\}|<\b$.
\end{quote}

We briefly describe how to obtain this family. Let $\{y_\alpha:\alpha<\d\}\subset[\omega]\sp{\omega}$ be a dominating family such that $\omega\setminus y_\alpha$ is infinite for all $\alpha<\d$. Assume that we have chosen $\{x_\alpha:\alpha<\beta\}$. Let $z\in[\omega]\sp{\omega}$ be a bound of $\{x_\alpha,y_\alpha,\omega\setminus y_\alpha:\alpha<\beta\}$. The set of all infinite, coinfinite $x\subset\omega$ such that $x\not\leq\sp\ast z$ and $\omega\setminus x\not\leq\sp\ast z$ is comeager in $[\omega]\sp{\omega}$, so choose $x_\beta$ with these properties. From this construction, property $(\ast)$ is easily seen to hold.

Define $x(\alpha,0)=x_\alpha$ and $x(\alpha,1)=\omega\setminus x_\alpha$ for all $\alpha<\b$. Then it is possible to choose $t\in{}\sp{\b}{2}$ such that $B=\{x(\alpha,t(\alpha)):\alpha<\b\}$ generates a proper ideal (that is, the union of finitely many elements of $B\cup[\omega]\sp{<\omega}$ does not cover $\omega$).

The next step is to prove that $B\cup[\omega]\sp{<\omega}$ has all its finite powers Menger. This is a non-trivial step but is exactly the same as the proof that the set constructed in Theorem 16 of \cite{bart-tsaban} has all its finite powers Menger.

Let $\I$ is the ideal generated by $B\cup[\omega]\sp{<\omega}$. Then by Lemma \ref{lemmabase} we have that $\I$ is Menger. To see that $\I$ is not Hurewicz, consider the dual filter $\I\sp\ast=\{\omega\setminus A:A\in\I\}$. Naturally, $B\sp\prime=\{x(\alpha,1-t(\alpha)):\alpha<\b\}\subset\I\sp\ast\subset[\omega]\sp\omega$. So we may think that $\I\sp\ast$ is a subset of $\baire$. Notice that implies that $B\sp\prime$ is unbounded. Thus, $\I\sp\ast$ is unbounded in $\baire$. By Proposition \ref{classichurewicz}, we find that $\I\sp\ast$ is not Hurewicz.
\end{proof}

\begin{coro}
There exists a Menger filter that is not Hurewicz in ZFC.
\end{coro}
\begin{proof}
If $\b<\d$, use any non-Hurewicz filter of size $\b$, for example the one constructed on Proposition \ref{lessthanb}. By Proposition \ref{lessthand}, this filter is Menger. If $\b=\d$, use the filter from Proposition \ref{bequalsd}.
\end{proof}

In fact it is still an open problem to find a set of reals of size $\d$ that is not Hurewicz and with all its finite powers Menger (see Problem 5.7 in \cite{tsaban-menger_hurewicz_book}). Thus, Question \ref{mengernonhurewicz} remains unsolved.

Let us remark that it is also known that if $\d$ is regular, then there is a non-Hurewicz set of reals of size $\d$ with all its powers Menger \cite[Theorem 9.1]{tsaban-zdomskyy-scales_fields_hurewicz}.

\section{FUF filters are Hurewicz}

In this section we present a result of Chodounsk\'y that FUF filters are Hurewicz (we would like to thank him for allowing us to include this result). This class of filters was introduced by Reznichenko and Sipacheva motivated by  Malykhin's question whether separable Fr\'echet topological groups are metrizable (see \cite{Rezn-Sip}).

The set $[\omega]\sp{<\omega}$ is a (Boolean) group under the symmetric difference and it is possible to define a topological group using filters on $\csp$. Let $\F$ be a filter on $\csp$. The idea is to make $\{[F]\sp{<\omega}:F\in\F\}$ a base at $\emptyset$. A $\pi$-network with respect to $\F$ is a collection $\X\subset[\omega]\sp{<\omega}$ such that for every $F\in\F$ there is $x\in\X$ with $x\subset F$. If $\S\subset\X$, we will say that $\S$ converges (to $\emptyset$) if for every $F\in\F$ the set of $x\in X$ with $x\not\subset F$ is finite. A filter $\F$ is FUF (Frech\'et-Urysohn for finite sets) if every time $\X\subset[\omega]\sp{<\omega}$ is a $\pi$-network wrt $\F$ there exists $\S\subset\X$ that $\F$-converges.

In fact, our motivation for considering these filters comes from Nyikos's proof that if $X$ is $\Sone{\Omega,\Gamma}$ (usually called a $\gamma$-set), then $\F_X$ is FUF \cite{nyikos-cantortree_fu}. From Theorem \ref{branches} we know that this type of filters are Hurewicz. However, we have the following: 

\begin{propo} (David Chodounsk\'y)
Every FUF filter is Hurewicz.
\end{propo}
\begin{proof}
Let $\{\U_n:n<\omega\}$ be a sequence of open covers of a FUF filter $\F$. For each $n<\omega$, let us define a cover $\V_n$ derived from $\U_n$. 

Given $F\in\F$ and $n<\omega$, since $\up{F}=\bigcap\{\up{(F\cap k)}\ :k<\omega\}$, there is $m<\omega$ such that $\up{(F\cap m)}$ is covered by finitely many elements of $\U_n$. So there exists a countable subfamily of $\{\up{x}:x\in[\omega]\sp{<\omega}\}$ that covers $\F$ and refines finite unions of $\U_n$; call this family $\V_n$.

Let $\X_n=\{x:\up{x}\in\V_n\}$ for all $n<\omega$. The fact that $\V_n$ is a cover of $\F$ immediately translates to the statement that $\X_n$ is a $\pi$-network wrt $\F$.

Let $\X_n=\{x(n,i):i<\omega\}$ be an enumeration for all $n<\omega$. Let $S=\{s\in[\omega]\sp{<\omega}:|s|=s(0)\}$ and for each $s\in S$ let $Y_s=\bigcup\{x(i,s(i)):i<|s|\}$. Define $\Y=\{Y_s:s\in S\}$. It is not hard to prove that $\Y$ is a $\pi$-network wrt $\F$. So there exists $\S\subset\Y$ that $\F$-converges. Enumerate $\S=\{Y_{s(i)}:i<\omega\}$.

Assume that there is $m<\omega$ such that $s_i(0)=m$ for infinitely many $i<\omega$. This implies that $x(0,m)\in F$ for all $F\in\F$, that is impossible.

Thus, for each $n<\omega$, the set of all those $i<\omega$ such that $s_i(0)=n$ is finite. Then we may choose $A\in[\omega]\sp\omega$ such that $s_i(0)<s_j(0)$ whenever $i,j\in A$ and $i<j$. Since clearly every subsequence of an $\F$-convergent sequence is $\F$-convergent, let us assume that $A=\omega$.

Define $t\in\omega\sp\omega$ such that given $i<\omega$, if $s_{j-1}<i\leq s_j$ for some $j<\omega$ (and $s_{-1}=-1$) then $t(i)=s_j(i)$. Then it is not hard to see that $\{x(n,t(n)):n<\omega\}$ is $\F$-convergent. This immediately translates the statement that $\X=\{\up{x(n,t(n))}:n<\omega\}$ is a $\gamma$-cover of $\F$. Finally, for each $n<\omega$, let $\W_n\in[\U_n]\sp{<\omega}$ be such that $\up{x(n,t(n))}\subset\bigcup\W_n$ (this follows from the definition of $\V_n$). Then $\{\bigcup\W_n:n<\omega\}$ is a $\gamma$-cover of $\F$, which is sufficient to finish the proof.
\end{proof}

\end{document}